\documentclass[draft]{amsart}
\usepackage{amsmath}
\usepackage{amssymb}
\usepackage{graphicx}
\newtheorem{theorem}{Theorem}[section]
\newtheorem{lemma}[theorem]{Lemma}

\theoremstyle{definition}

\newtheorem{example}[theorem]{Example}

\newtheorem{remark}[theorem]{Remark}

\numberwithin{equation}{section}

\begin{document}

\title[A note on the quantization error for ISMs]{A note on the quantization error for in-homogeneous self-similar measures}

\author{Sanguo Zhu}
\address{School of Mathematics and Physics, Jiangsu University
of Technology\\ Changzhou 213001, China.}
\email{sgzhu@jsut.edu.cn}


\begin{abstract}
We further study the asymptotics of quantization errors for two classes of in-homogeneous self-similar measures $\mu$. We give a new sufficient condition for the upper quantization coefficient for $\mu$ to be finite. This, together with our previous work, leads to a necessary and sufficient condition for the upper and lower quantization coefficient of $\mu$ to be both positive and finite. Furthermore, we determine (estimate) the convergence order of the quantization error in case that the quantization coefficient is infinite.
\end{abstract}

\maketitle

\section{Introduction}

Let $(f_i)_{i=1}^N$ be a family of contractive similitudes on $\mathbb{R}^q$. By \cite{Hut:81}, there exists a unique non-empty compact set $E$ satisfying
\[
E=f_1(E)\cup\cdots\cup f_N(E).
\]
This set is called the self-similar set associated with $(f_i)_{i=1}^N$. Given a probability $(q_i)_{i=1}^N$, there exists a unique Borel probability measure $P$ supported on $E$ with
\begin{eqnarray}\label{ssm}
P=\sum_{i=1}^Nq_iP\circ f_i^{-1}.
\end{eqnarray}
The measure $P$ is called the self-similar measure associated with $(f_i)_{i=1}^N$ and the probability vector $(q_i)_{i=1}^N$.

Let $\nu$ be a Borel probability measure on $\mathbb{R}^q$ with compact support  $C$ and $(p_i)_{i=0}^N$ a probability vector with $p_i>0$ for all $0\leq i\leq N$. By \cite{Bar:88,Olsen:08}, there exists a unique a Borel probability measure $\mu$ satisfying
\begin{eqnarray}\label{s55}
\mu=p_0\nu+\sum_{i=1}^Np_i\mu\circ f_i^{-1}.
\end{eqnarray}
We call the measure $\mu$ the in-homogeneous self-similar measure (ISM) associated with $(f_i)_{i=1}^N,(p_i)_{i=0}^N$ and $\nu$. The support $K$ of $\mu$ is the unique nonempty compact set satisfying
\begin{eqnarray}\label{s15}
K=C\cup f_1(K)\cup\cdots\cup f_N(K).
\end{eqnarray}
Without loss of generality, we always assume that the diameter of $K$ equals $1$. We further consider the following two disjoint classes of ISMs.

\emph{Case (i)}: Let $(f_i)_{i=1}^N$ satisfy the open set condition (OSC), namely, there exists a non-empty bounded open set $U$ such that
$f_i(U),1\leq i\leq N$, are pairwise disjoint and $\bigcup_{i=1}^Nf_i(U)\subset U$. Let $\nu$  a self-similar measure associated with $(f_i)_{i=1}^N$ and a probability vector $(t_i)_{i=1}^N$ with $t_i>0$  for all $1\leq i\leq N$. Then $C=E$; by (\ref{s15}), we have that $K=E$.

\emph{Case (ii)}: Let $(g_i)_{i=1}^M$ be a family of contractive similitudes satisfying the OSC with contraction ratios $(c_i)_{i=1}^M$. Let $\nu$ be the self-similar measure associated with $(g_i)_{i=1}^M$ and a probability vector $(t_i)_{i=1}^M$ with $t_i>0$  for all $1\leq i\leq M$. Let ${\rm cl}(A),\partial(A)$ and ${\rm int}(A)$ respectively denote the closure, boundary and interior in $\mathbb{R}^q$ of a set $A$. We assume a modified version of the in-homogeneous open set condition (IOSC) proposed in \cite{Olsen:08}: there exists a bounded non-empty open set $U$ such that
\begin{itemize}
\item[\rm (1)] $\bigcup_{i=1}^Nf_i(U)\subset U$ and $f_i(U)\cap f_j(U)=\emptyset,1\leq i\neq j\leq N$;

\item[\rm (2)] $E\cap U\neq\emptyset$ and $C\subset U$;

\item[\rm (3)] $\nu(\partial(U))=0$; $C\cap f_i({\rm cl}(U))=\emptyset$ for all $1\leq i\leq N$.
\end{itemize}

Let $(s_i)_{i=1}^N$ be the contraction ratios of $(f_i)_{i=1}^N$.
 For Case (i), we define two positive numbers $\xi_{1,r},\xi_{2,r}$  by
 \[
 \sum_{i=1}^N(t_is_i^r)^{\frac{\xi_{1,r}}{\xi_{1,r}+r}}=1;\;\;\sum_{i=1}^N(p_is_i^r)^{\frac{\xi_{2,r}}{\xi_{2,r}+r}}=1.
 \]
 As no confusion could arise, we define for Case (ii), two positive numbers, which we still denote by $\xi_{1,r},\xi_{2,r}$, by
 \[
 \sum_{i=1}^M(t_ic_i^r)^{\frac{\xi_{1,r}}{\xi_{1,r}+r}}=1;\;\sum_{i=1}^N(p_is_i^r)^{\frac{\xi_{2,r}}{\xi_{2,r}+r}}=1.
 \]

In the present paper, we will further examine the finiteness of the upper quantization coefficient for $\mu$. We refer to \cite{GL:00} for mathematical foundations of quantization theory and \cite{GN:98} for its deep background in information theory. One may see \cite{GL:01,GL:04,GL:05,PK:01} for more related results.

For a Borel probability measure $P$, the $s$-dimensional upper and lower quantization coefficient are defined by
\begin{eqnarray*}
\overline{Q}_r^s(P):=\limsup_{n\to\infty}n^{\frac{r}{s}}e_{n,r}^r(P),\;\;
\underline{Q}_r^s(P):=\liminf_{n\to\infty}n^{\frac{r}{s}}e^r_{n,r}(P),
\end{eqnarray*}
where $e_{n,r}(P)$ is the error in the approximation of $P$ with discrete probability measures supported on at most $n$ points in the sense of $L_r$-metrics. Set
\[
\mathcal{D}_{k}:=\{\alpha\subset\mathbb{R}^{q}:1\leq{\rm card}(\alpha)\leq k\},\;k\geq 1.
\]
Then by \cite{GL:00}, we have the following equivalent definition for $e_{n,r}(P)$:
\[
e_{n,r}(P):=\inf_{\alpha\in\mathcal{D}_n}\bigg(\int d(x,\alpha)^rdP(x)\bigg)^{\frac{1}{r}}.
\]
By \cite[Theorem 4.12]{GL:00}, $e_{n,r}(P)$ is strictly decreasing with respect to $n$ provided that
\[
{\rm card}({\rm supp}(P))=\infty\;{\rm and}\;\int|x|^rdP(x)<\infty.
\]

The upper (lower) quantization dimension $\overline{D}_r(P)$ ($\underline{D}_r(P)$) for $P$ of order $r$ is exactly the critical point at which the upper (lower) quantization coefficient jumps from zero to infinity (cf. \cite{GL:00,PK:01}):
\begin{eqnarray*}
\overline{D}_r(P):=\limsup_{n\to\infty}\frac{\log n}{-\log e_{n,r}(P)},\;\underline{D}_r(P):=\liminf_{n\to\infty}\frac{\log n}{-\log e_{n,r}(P)}.
\end{eqnarray*}
Both the upper (lower) dimension and the upper (lower) quantization coefficient are characterizations of the asymptotic properties of the quantization errors, while the latter provides us with more accurate information.

 Although the supports and mass distributions of the ISMs in the above two cases are completely different (see (\ref{s12}) and (\ref{gg2})), these ISMs share many properties concerning the asymptotic quantization errors. As is proved in \cite{zhu:14,zhu:13}, for an ISM in Case (i) or (ii), we have
\begin{eqnarray}
 \underline{D}_r(\mu)=\overline{D}_r(\mu)=\xi_r:=\max\{\xi_{1,r},\xi_{2,r}\};\label{gg3}\\
 \underline{Q}_r^{\xi_r}(\mu)>0;\;\;\overline{Q}_r^{\xi_r}(\mu)<\infty\;{\rm if}\;\xi_{1,r}>\xi_{2,r}\label{gg4}.
\end{eqnarray}
It was left open whether the $\overline{Q}_r^{\xi_r}(\mu)<\infty$ when $\xi_{1,r}<\xi_{2,r}$. We will prove

\begin{theorem}\label{mthm2}
Let $\mu$ be an ISM in Case (i) or (ii). Then $\overline{Q}_r^{\xi_r}(\mu)<\infty$ if $\xi_{1,r}<\xi_{2,r}$.
\end{theorem}

For two number sequences $(a_n)_{n=1}^\infty$ and $(b_n)_{n=1}^\infty$, we write $a_n\lesssim b_n$ ($a_n\gtrsim b_n$) if there exists a constant $C>0$ such that $a_n\leq C b_n$ ($a_n\geq C b_n$) for all large $n$; we write $a_n\asymp b_n$ if $a_n\lesssim b_n$ and $a_n\gtrsim b_n$.
By Theorem \ref{mthm2}, (\ref{gg3}) and (\ref{gg4}), for an ISM $\mu$ in Case (i) or (ii), we have, $\overline{Q}_r^{\xi_r}(\mu)<\infty$ if and only if $\xi_{1,r}\neq\xi_{2,r}$. As a consequence, we have
\[
e^r_{n,r}(\mu)\asymp n^{-\frac{r}{\xi_r}}\;\;{\rm if}\;\;\xi_{1,r}\neq\xi_{2,r}.
\]
For the cases when $\xi_{1,r}=\xi_{2,r}$, we will show

\begin{theorem}\label{mthm3}
Assume that $\xi_{1,r}=\xi_{2,r}$. Then
for an ISM $\mu$ in Case (i), we have
\begin{eqnarray}\label{sanguo1}
n^{-\frac{r}{\xi_r}}\cdot\log n\lesssim e^r_{n,r}(\mu)\lesssim n^{-\frac{r}{\xi_r}}\cdot(\log n)^{\frac{\xi_r+r}{\xi_r}};
\end{eqnarray}
for an ISM $\mu$ in Case (ii), we have
\begin{eqnarray}\label{sanguo}
e_{n,r}^r(\mu)\asymp n^{-\frac{r}{\xi_r}}(\log n)^{\frac{\xi_r+r}{\xi_r}}.
\end{eqnarray}
\end{theorem}

At the end of the paper, we will construct concrete examples to illustrate our main result. In contrast to self-similar measures (cf. \cite[Theorem 3.1]{GL:01}), our examples also show that, the upper quantization coefficient for an ISM $\mu$ of order $r$ can be finite for some $r$ while infinite for another $r$.

\section{Proofs of main results}
\subsection{Proof of Theorem \ref{mthm2}: Case (i)}
We will need the following notations. Set
\begin{eqnarray*}
&&\Omega_n:=\{1,\ldots, N\}^n,\; n\geq 1;\;\Omega^*:=\bigcup_{n=1}^\infty\Omega_n;\;\;\Omega_\infty:=\{1,\ldots, N\}^{\mathbb{N}};\\&&s_\sigma:=\prod_{h=1}^ns_{\sigma_h},\;p_\sigma:=\prod_{h=1}^np_{\sigma_h},\;
t_\sigma:=\prod_{h=1}^nt_{\sigma_h},\;\sigma=(\sigma_1,\ldots,\sigma_n)\in\Omega_n.
\end{eqnarray*}
For $\sigma\in\Omega_n$, we define $|\sigma|:=n$ and $\sigma|_0=\theta:=$empty word. For
$1\leq h<n$ and $\sigma\in\Omega^*$ with $|\sigma|\geq h$, we
write
\begin{eqnarray*}
\sigma^{(l)}_{-h}:=(\sigma_{h+1},\ldots,\sigma_{|\sigma|}),\;\sigma|_h:=(\sigma_1,\ldots,\sigma_h),\;\;\sigma^-:=\sigma|_{|\sigma|-1}.
\end{eqnarray*}
As is shown in \cite[Lemma 2.1]{zhu:13}, we have
\begin{eqnarray}\label{s12}
\mu(E_\sigma)=\sum_{h=0}^{k-1}p_0p_{\sigma|_h} t_{\sigma_{-h}^{(l)}}+p_\sigma,\;\;\sigma\in\Omega_k,\;\;k\geq 1.
\end{eqnarray}
For $\sigma,\tau\in\Omega^*$, we write
$\sigma\ast\tau:=(\sigma_1,\ldots,\sigma_{|\sigma|},\tau_1,\ldots,\tau_{|\tau|})$.

If $\sigma,\tau\in\Omega^*$ and
$|\sigma|\leq|\tau|,\sigma=\tau|_{|\sigma|}$, then we write $\sigma\preceq\tau$ and call $\sigma$ a predecessor of $\tau$.
A finite set $\Gamma\subset\Omega^*$ is called an antichain if for any two words $\sigma,\tau\in\Gamma$, we have neither $\sigma\preceq\tau$ nor $\tau\preceq\sigma$; a finite antichain $\Gamma$ is said to be maximal if every $\tau\in\Omega_\infty$ has a predecessor in $\Gamma$.
We write
\begin{eqnarray*}
&&h(\sigma):=\mu(E_\sigma) s_\sigma^r;\;\underline{\eta}_r:=\min_{1\leq i\leq N}\min\big\{(p_0t_i+p_i),t_i\big\}s_i^r;\\&&\Lambda_{k,r}:=\{\sigma\in\Omega^*:h(\sigma^-)\geq k^{-1}\underline{\eta}_r>h(\sigma)\},\;\;N_{k,r}:={\rm card}(\Lambda_{k,r}).
\end{eqnarray*}
 One can see that $\Lambda_{k,r},k\geq 1$, are finite maximal antichains.

  For an ISM $\mu$ in Case (i), by \cite[Lemma 2.2]{zhu:13}, we have
\begin{eqnarray}\label{knownestimate}
D\sum_{\sigma\in\Lambda_{k,r}}h(\sigma)\leq e^r_{ N_{k,r},r}(\mu)\leq\sum_{\sigma\in\Lambda_{k,r}}h(\sigma),
\end{eqnarray}
where $D>0$, is a constant which is independent of $k$. We write
\begin{eqnarray}\label{lambdak}
&&\lambda_{k,r}:=\sum_{\sigma\in\Lambda_{k,r}}h(\sigma)^{\frac{\xi_r}{\xi_r+r}},\;\;\overline{\lambda}:=\limsup_{k\to\infty}\lambda_{k,r},
\;\;\underline{\lambda}:=\liminf_{k\to\infty}\lambda_{k,r};\\
&&Q_{k,r}:=N_{k,r}^{\frac{r}{\xi_r}}e_{N_{k,r},r}(\mu),\;\;\overline{P}_r^{\xi_r}(\mu):=\limsup_{k\to\infty}Q_{k,r}.\nonumber
\end{eqnarray}

\begin{lemma}\label{z1}
We have $\lambda_{k,r}\asymp Q_{k,r}^{\frac{\xi_r}{\xi_r+r}}$. As a result, we have, $\overline{Q}_r^{\xi_r}(\mu)<\infty$ if and only if $\overline{\lambda}<\infty$; $\underline{Q}_r^{\xi_r}(\mu)>0$ if and only if $\underline{\lambda}>0$.
\end{lemma}
\begin{proof}
By (\ref{knownestimate}), for all large $k$, we have
\begin{eqnarray*}
\sum_{\sigma\in\Lambda_{k,r}}h(\sigma)\leq D^{-1}e^r_{N_{k,r},r}(\mu).
\end{eqnarray*}
Thus, by H\"{o}lder's inequality, for all large $k$, we have
\begin{eqnarray*}
\lambda_{k,r}&=&\sum_{\sigma\in\Lambda_{k,r}}h(\sigma)^{\frac{\xi_r}{\xi_r+r}}\leq \bigg(\sum_{\sigma\in\Lambda_{k,r}}h(\sigma)\bigg)^{\frac{\xi_r}{\xi_r+r}}N_{k,r}^{\frac{r}{\xi_r+r}}\\
&\leq&\bigg(D^{-1}e^r_{N_{k,r},r}(\mu)\bigg)^{\frac{\xi_r}{\xi_r+r}}\cdot N_{k,r}^{\frac{r}{\xi_r+r}}\asymp Q_{k,r}^{\frac{\xi_r}{\xi_r+r}}.
\end{eqnarray*}
It follows that $\lambda_{k,r}^{\frac{\xi_r+r}{\xi_r}}\lesssim Q_{k,r}$. On the other hand, by (\ref{lambdak}) and the definition of $\Lambda_{k,r}$, one can see that $N_{k,r}(k^{-1}\underline{\eta}_r^2)^{\frac{\xi_r}{\xi_r+r}}\leq \lambda_{k,r}$. It follows that
\begin{eqnarray*}
e^r_{ N_{k,r},r}(\mu)&\leq& \sum_{\sigma\in\Lambda_{k,r}}h(\sigma)\leq \sum_{\sigma\in\Lambda_{k,r}}h(\sigma)^{\frac{\xi_r}{\xi_r+r}}\cdot h(\sigma)^{\frac{r}{\xi_r+r}}\leq \lambda_{k,r} (k^{-1}\underline{\eta}_r)^{\frac{r}{\xi_r+r}}\\
&\leq&\lambda_{k,r}(\lambda_{k,r}N_{k,r}^{-1})^{\frac{r}{\xi_r}}\underline{\eta}_r^{-\frac{r}{\xi_r+r}}=N_{k,r}^{-\frac{r}{\xi_r}}\underline{\eta}_r^{-\frac{r}{\xi_r+r}}\lambda_{k,r}^{1+\frac{r}{\xi_r}}.
\end{eqnarray*}
Hence $\lambda_{k,r}^{\frac{\xi_r+r}{\xi_r}}\gtrsim Q_{k,r}$. Combining the above analysis, the first part of the lemma follows. This and \cite[Lemma 3.6]{zhu:13} imply the remaining part.
\end{proof}
\begin{remark}
In view of (\ref{s12}), ISMs in Case (i) are typically not Markov-type measures. However, for the proof of Theorem \ref{mthm2} Case (i),we will benefit from \cite[Proposition 3.13]{KZ:14} for some ideas of classifying words in $\Lambda_{k,r}$, while the techniques we used in \cite{zhu:13} unfortunately does not work.
\end{remark}
\emph{Proof of Theorem \ref{mthm2}: Case (i)}

For every $k\geq 1$, we have
\begin{eqnarray}
\lambda_{k,r}&=&\sum_{\sigma\in\Lambda_{k,r}}\bigg(\sum_{h=0}^{|\sigma|-1}(p_0p_{\sigma|_h} t_{\sigma_{-h}^{(l)}})s_\sigma^r+p_\sigma s_\sigma^r\bigg)^{\frac{\xi_r}{\xi_r+r}}\nonumber
\\&\leq&\sum_{\sigma\in\Lambda_{k,r}}\sum_{h=0}^{|\sigma|-1}(p_0p_{\sigma|_h} t_{\sigma_{-h}^{(l)}}s_\sigma^r)^{\frac{\xi_r}{\xi_r+r}}+\sum_{\sigma\in\Lambda_{k,r}}(p_\sigma s_\sigma^r)^{\frac{\xi_r}{\xi_r+r}}.\label{estimate00}
\end{eqnarray}
For each $h\geq 1$ and $\rho\in\Omega_h$, we write
\[
\Gamma(\rho):=\{\omega\in\Omega^*:\omega\ast\rho\in\Lambda_{k,r}\}.
\]
One can see that $\Gamma(\rho)$ may be empty for some $\rho\in\Omega^*$. Set
\[
\overline{\eta}_r:=\max_{1\leq i\leq N}\max\big\{(p_0t_i+p_i),t_i\big\}s_i^r.
\]
Let $H$ be the smallest integer such that $\overline{\eta}_r^{H}<2^{-1}$. For $\omega\in\Gamma(\rho)$ and $\tau\in\Omega_{H}$,
\[
h(\omega\ast\tau\ast\rho)<k^{-1}\underline{\eta}_r\cdot\overline{\eta}_r^{H}<k^{-1}\underline{\eta}_r\cdot 2^{-1}\leq (k+1)^{-1}\underline{\eta}_r.
\]
 This implies that, $\big||\omega|-|\widetilde{\omega}|\big|\leq H$ for every pair $\omega,\widetilde{\omega}\in\Gamma(\rho)$ with $\omega\preceq\widetilde{\omega}$ or $\widetilde{\omega}\preceq\omega$. For every $\sigma\in\Omega^*$ and $j\geq 0$, we write
\begin{eqnarray}\label{s30}
\Gamma(\sigma,j):=\{\tau\in\Omega_{|\sigma|+j}:\sigma\preceq\tau\}.
\end{eqnarray}
Then, by the above analysis, there exists a finite antichain $\Upsilon(\rho)$
such that
\begin{eqnarray}\label{gg1}
\Gamma(\rho)\subset\bigcup_{\tau\in\Upsilon(\rho)}\bigcup_{j=1}^{H}\Gamma(\tau,j).
\end{eqnarray}
By the hypothesis, we have, $\xi_r=\xi_{2,r}>\xi_{1,r}$. Hence,
\[
b(\xi_r):=\sum_{i=1}^N(p_i s_i^r)^{\frac{\xi_r}{\xi_r+r}}=1.
\]
Furthermore, we have $\sum_{\sigma\in\Upsilon}(p_\sigma s_\sigma^r)^{\frac{\xi_r}{\xi_r+r}}=1$ for every finite maximal antichain $\Upsilon$. Using this and (\ref{gg1}), we deduce
\begin{eqnarray}\label{estimate01}
&& \sum_{\omega\in\Gamma(\rho)}(p_\omega s_\omega^r)^{\frac{\xi_r}{\xi_r+r}}\leq\sum_{\tau\in\Upsilon(\rho)}\sum_{j=1}^{H}\sum_{\sigma\in\Gamma(\tau, j)}(p_\sigma s_\sigma^r)^{\frac{\xi_r}{\xi_r+r}}\nonumber\\
  &&\leq\sum_{\tau\in\Upsilon(\rho)}(p_\tau s_\tau^r)^{\frac{\xi_r}{\xi_r+r}}\bigg(1+\sum_{j=1}^{H}\bigg(\sum_{i=1}^N(p_i s_i^r)^{\frac{\xi_r}{\xi_r+r}}\bigg)^h\bigg)=1+H.
\end{eqnarray}
Set $l_{1j}:=\min_{\sigma\in\Lambda_{k,r}}|\sigma|$ and $l_{2j}:=\max_{\sigma\in\Lambda_{k,r}}|\sigma|$.
We classify the words in $\Lambda_{k,r}$ according to the suffices $\sigma^{(l)}_{-h}$. Note that for every $\sigma\in\Lambda_{k,r}$, we have $l_{1j}\leq|\sigma|\leq l_{2j}$ and $\sigma^{(l)}_{-h}\in\Omega_{|\sigma|-h},h\leq|\sigma|-1$. We have
\[
\Lambda_{k,r}\subset\bigcup_{n=1}^{l_{2j}}\bigcup_{h=1}^{n-1}\bigcup_{\rho\in\Omega_{n-h}}\Gamma(\rho)=\bigcup_{h=1}^{l_{2j}-1}
\bigcup_{\rho\in\Omega_h}\Gamma(\rho).
\]
Since $\Lambda_{k,r}$ is a maximal antichain, we have, $\sum_{\sigma\in\Lambda_{k,r}}(p_\sigma s_\sigma^r)^{\frac{\xi_r}{\xi_r+r}}=1$. Using (\ref{estimate00}) and (\ref{estimate01}), we further deduce
\begin{eqnarray*}
\lambda_{k,r}&\leq&\sum_{\sigma\in\Lambda_{k,r}}\sum_{h=0}^{|\sigma|-1}(p_{\sigma|_h} t_{\sigma_{-h}^{(l)}}s_\sigma^r)^{\frac{\xi_r}{\xi_r+r}}+1\\
&\leq&\sum_{h=1}^{l_{2j}-1}\sum_{\rho\in\Omega_h} \sum_{\omega\in\Gamma(\rho)}(p_\omega t_\rho s_\omega s_\rho^r)^{\frac{\xi_r}{\xi_r+r}}+1\\
&\leq& (1+H)\sum_{h=1}^{l_{2j}-1}\sum_{\rho\in\Omega_h}(t_\rho s_\rho^r)^{\frac{\xi_r}{\xi_r+r}}+1.
\end{eqnarray*}
Again, by the hypothesis, we have, $a(\xi_r):=\sum_{i=1}^N(t_i s_i^r)^{\frac{\xi_r}{\xi_r+r}}<1$. Hence,
\begin{eqnarray*}
\lambda_{k,r}\leq (1+H)\sum_{h=1}^{l_{2j}-1}a(\xi_r)^h+1\leq\frac{(1+H)a(\xi_r)}{1-a(\xi_r)}+1<\infty.
\end{eqnarray*}
Thus, by Lemma \ref{z1}, we conclude that $\overline{Q}_r^{\xi_r}(\mu)$ is finite.

\subsection{Proof of Theorem \ref{mthm2}: Case (ii)}
In the following, we consider ISMs in Case (ii). Write
\[
\Phi_n:=\{1,\ldots, M\}^n, \;\Phi^*:=\bigcup_{n=1}^\infty\Phi_n,\;c_\sigma:=\prod_{h=1}^nc_{\sigma_h}.
\]
All the notations for words in $\Phi^*$ are defined in the same way as for words in $\Omega^*$. Let $\Gamma(\sigma,j)$ be as defined in (\ref{s30}). For every $\sigma\in\Omega^*$, we write
\begin{eqnarray*}
\Gamma^*(\sigma):=\bigcup_{j\geq 1}\Gamma(\sigma,j).
\end{eqnarray*}
For a finite maximal antichain $\Upsilon\subset\Omega^*$, we define
\[
l(\Upsilon):=\min_{\rho\in\Upsilon}|\rho|,\;\;L(\Upsilon):=\max_{\rho\in\Upsilon}|\rho|.
\]
For each $\sigma\in\Omega_{l(\Upsilon)}$, we define
\begin{eqnarray*}
\Lambda_\Upsilon(\sigma):=\{\tau\in\Omega^*:\sigma\preceq\tau,\Gamma^*(\tau)\cap\Upsilon\neq\emptyset\},\;\;\Lambda_\Upsilon^*:=\bigcup_{\sigma\in\Omega_{l(\Upsilon)}}\Lambda_\Upsilon(\sigma).
\end{eqnarray*}
Assuming (1)-(3), by \cite[lemma 2.2]{zhu:14}, for every $\sigma\in\Omega^*$ and $\omega\in\Phi^*$, we have
\begin{eqnarray}\label{gg2}
\mu(f_\sigma(K))=p_\sigma,\;\mu(f_\sigma(C_\omega))=p_0p_\sigma t_\omega.
\end{eqnarray}
We will work with the following notations (cf. \cite{zhu:13}):
\begin{eqnarray*}
&&\underline{\pi}_r:=\min\big\{\min_{1\leq i\leq N}p_is_i^r,\min_{1\leq i\leq M}t_ic_i^r\big\},\;\overline{\pi}_r:=\max\big\{\max_{1\leq i\leq N}p_is_i^r,\max_{1\leq i\leq M}t_ic_i^r\big\};\\&&\Gamma_{k,r}:=\{\sigma\in\Omega^*:p_{\sigma^-}c_{\sigma^-}^r\geq\underline{\pi}_r^k>p_\sigma c_\sigma^r\};\;\;\widetilde{l}_{1k}:=\min_{\sigma\in\Gamma_{k,r}}|\sigma|,\;\widetilde{l}_{2k}:=\max_{\sigma\in\Gamma_{k,r}}|\sigma|;
\\&&\Gamma_{k,r}(\sigma):=\{\rho\in\Phi^*:p_\sigma c_\sigma^rt_{\rho^-}c_{\rho^-}^r\geq\underline{\pi}_r^k>p_\sigma c_\sigma^rt_\rho c_\rho^r\},\;|\sigma|\leq l_{1k}-1;\\&&\Psi_{k,r}:=\bigcup_{h=0}^{\widetilde{l}_{1k}-1}\Omega_h\cup\Lambda^*_{\Gamma_{k,r}};\;\phi_{k,r}:={\rm card}(\Gamma_{k,r})+{\rm card}(\Psi_{k,r});\;\widetilde{Q}_{k,r}:=\phi_{k,r}^{\frac{r}{\xi_r}}e^r_{\phi_{k,r},r}(\mu).
\end{eqnarray*}
By \cite[Lemmas 3.3, 4.10]{zhu:14}, for an ISM  $\mu$ in Case (ii), we have
\begin{eqnarray}\label{zsg11}
e^r_{\phi_{k,r},r}(\mu)\asymp \widetilde{\lambda}_{k,r}:=\sum_{\sigma\in\Psi_{k,r}}\sum_{\rho\in\Gamma_{k,r}(\sigma)}(p_\sigma s_\sigma^r t_\rho c_\rho^r)^{\frac{\xi_r}{\xi_r+r}}+\sum_{\sigma\in\Gamma_{k,r}}(p_\sigma s_\sigma^r)^{\frac{\xi_r}{\xi_r+r}}.
\end{eqnarray}

\emph{Proof of Theorem \ref{mthm2}: Case (ii)}

Let $\widetilde{H}$  be the smallest integer such that
$\big(\max_{1\leq i\leq N}p_is_i^r\big)^{\widetilde{H}}<\underline{\pi}_r$.
Write
\begin{eqnarray}\label{gg777}
\Lambda(\rho):=\{\omega\in\Omega^*:p_\sigma c_\sigma^rt_{\rho^-}c_{\rho^-}^r\geq\underline{\pi}_r^k>p_\sigma c_\sigma^rt_\rho c_\rho^r\},\;\rho\in\Phi^*.
\end{eqnarray}
As in Case (i), there exists some finite maximal antichain $\widetilde{\Upsilon}(\rho)$ in $\Omega^*$ such that (\ref{gg1}) and (\ref{estimate01}) hold with $\Lambda(\rho),\widetilde{\Upsilon}(\rho),\widetilde{H}$ in place of $\Gamma(\rho),\Upsilon(\rho),H$. We have
\[
\Psi_{k,r}=\bigcup_{h=1}^{l_{1k}-1}\Omega_h\cup\Lambda^*_{\Gamma_{k,r}}\subset\bigcup_{\rho\in\Phi^*}\Lambda(\rho).
\]
By the hypothesis, we have, $a(\xi_r):=\sum_{i=1}^M(t_\rho c_\rho^r)^{\frac{\xi_r}{\xi_r+r}}<1$. It follows that
\begin{eqnarray}\label{z2}
\widetilde{\lambda}_{k,r}&=&\sum_{\sigma\in\Psi_{k,r}}\sum_{\rho\in\Gamma_{k,r}(\sigma)}(p_\sigma s_\sigma^r t_\rho c_\rho^r)^{\frac{\xi_r}{\xi_r+r}}+\sum_{\sigma\in\Gamma_{k,r}}(p_\sigma s_\sigma^r)^{\frac{\xi_r}{\xi_r+r}}\nonumber\\&\leq&\sum_{\rho\in\Phi^*}\sum_{\sigma\in\Lambda(\rho)}(p_\sigma s_\sigma^r t_\rho c_\rho^r)^{\frac{\xi_r}{\xi_r+r}}+1\nonumber\\&\leq&(1+\widetilde{H})\sum_{\rho\in\Phi^*}(t_\rho c_\rho^r)^{\frac{\xi_r}{\xi_r+r}}+1=\frac{(1+\widetilde{H})a(\xi_r)}{1-a(\xi_r)}+1.
\end{eqnarray}
Using (\ref{zsg11}) and the proof of Lemma \ref{z1}, one show that
$\widetilde{Q}_{k,r}\asymp \widetilde{\lambda}_{k,r}^{\frac{\xi_r+r}{\xi_r}}$.
By \cite[Lemma 3.2]{zhu:14}, we have $\phi_{k,r}\asymp\phi_{k+1,r}$.
Hence, the theorem follows by (\ref{z2}).

\subsection{Proof of Theorem \ref{mthm3}}

 Let $\mu$ be an ISM in Case (i). Write
\[
a(s):=\sum_{i=1}^N(t_i s_i^r)^{\frac{s}{s+r}},\;\;b(s):=\sum_{i=1}^N(p_i s_i^r)^{\frac{s}{s+r}};\;\;s>0.
\]
By \cite[Lemma 3.4]{zhu:13}, for all $s\geq\xi_r$, we have
\begin{eqnarray}\label{zz1}
\sum_{\sigma\in\Lambda_{k,r}}h(\sigma)^{\frac{s}{s+r}}\leq\sum_{h=0}^{l_{1k}-1}a(s)^{l_{1k}-h}b(s)^{h}+\sum_{h=l_{1k}}^{l_{2k}}b(s)^h+b(s)^{l_{1k}}
\end{eqnarray}
By the hypothesis, $\xi_{1,r}=\xi_{2,r}$. Thus, we have
\[
\xi_r=\xi_{1,r}=\xi_{2,r};\;a(\xi_r)=b(\xi_r)=1.
\]
This, together with (\ref{zz1}), yields
\begin{eqnarray}\label{zz2}
\lambda_{k,r}=\sum_{\sigma\in\Lambda_{k,r}}h(\sigma)^{\frac{\xi_r}{\xi_r+r}}\leq l_{2k}+2.
\end{eqnarray}
On the other hand, it is proved in  \cite[Lemma 3.5]{zhu:13} that
\begin{eqnarray}\label{zz3}
\lambda_{k,r}=\sum_{\sigma\in\Lambda_{k,r}}h(\sigma)^{\frac{\xi_r}{\xi_r+r}}\geq (p_0l_{1k})^{\frac{\xi_r}{\xi_r+r}}.
\end{eqnarray}
Combining (\ref{zz1}), (\ref{zz2}) and the proof of Lemma \ref{z1}, we deduce
\begin{eqnarray}\label{ss1}
e^r_{N_{k,r},r}(\mu)\left\{ \begin{array}{ll}
\lesssim \lambda_{k,r}^{\frac{\xi_r+r}{\xi_r}}N_{k,r}^{-\frac{r}{\xi_r}}\leq l_{2k}^{\frac{\xi_r+r}{\xi_r}}N_{k,r}^{-\frac{r}{\xi_r}}\\
\gtrsim N_{k,r}^{-\frac{r}{\xi_r}}\lambda_{k,r}^{\frac{\xi_r+r}{\xi_r}}\gtrsim l_{1k}N_{k,r}^{-\frac{r}{\xi_r}}
\end{array}\right..
\end{eqnarray}
Also, by (\ref{zz1}), (\ref{zz2}) and the definition of $\Gamma_{k,r}$ , we have
\begin{eqnarray*}
l_{1k}\leq N_{k,r}\cdot(k^{-1}\underline{\eta}_r)^{\frac{\xi_r}{\xi_r+r}}\leq (l_{2k}+2)\cdot\underline{\eta}_r^{-\frac{\xi_r}{\xi_r+r}}\lesssim l_{2k}.
\end{eqnarray*}
Hence, $l_{1k}(k\underline{\eta}_r^{-1})^{\frac{\xi_r}{\xi_r+r}}\lesssim N_{k,r}\lesssim l_{2k}(k\underline{\eta}_r^{-1})^{\frac{\xi_r}{\xi_r+r}}$. By the definition of $l_{1k}$ and $l_{2k}$,
\begin{eqnarray*}
\underline{\eta}_r^{l_{1k}}<k^{-1}\underline{\eta}_r;\;(\overline{\eta}_r)^{l_{2k}-1}\geq k^{-1}\underline{\eta}_r^2.
\end{eqnarray*}
Hence, $l_{1k},l_{2k}\asymp \log k$. Combining the above analysis, we obtain
\begin{eqnarray*}
\log N_{k,r}\asymp \log \log k+\frac{\xi_r}{\xi_r+r}\log k\asymp \log k.
\end{eqnarray*}
 This, together with (\ref{ss1}), yields
\begin{eqnarray}\label{zz4}
N_{k,r}^{-\frac{r}{\xi_r}}\log N_{k,r}\lesssim e^r_{N_{k,r},r}(\mu)\lesssim N_{k,r}^{-\frac{r}{\xi_r}}(\log N_{k,r})^{\frac{\xi_r+r}{\xi_r}}.
\end{eqnarray}
As we noted \cite[Lemma 3.6]{zhu:13}, we have, $N_{k+1,r}\leq N N_{k,r}$. For every $n$, there exists a unique $k$ such that $N_{k,r}\leq n< N_{k+1,r}$. Hence,
\begin{eqnarray*}\label{zz4}
e^r_{n,r}(\mu)\left\{ \begin{array}{ll}
\leq e^r_{N_{k,r},r}(\mu)\lesssim N_{k,r}^{-\frac{r}{\xi_r}}(\log N_{k,r})^{\frac{\xi_r+r}{\xi_r}}\lesssim n^{-\frac{r}{\xi_r}}(\log n)^{\frac{\xi_r+r}{\xi_r}}\\
\geq e^r_{N_{k+1,r},r}(\mu)\gtrsim N_{k+1,r}^{-\frac{r}{\xi_r}}\log N_{k+1,r}\gtrsim n^{-\frac{r}{\xi_r}}\log n
\end{array}\right..
\end{eqnarray*}

 Next, let $\mu$ be an ISM in Case (ii). Note that $\sum_{\sigma\in\Gamma_{k,r}}(p_\sigma s_\sigma^r)^{\frac{\xi_r}{\xi_r+r}}=1$. By the proof of \cite[Lemma 3.5]{zhu:14}, we have
\begin{eqnarray}\label{ss2}
\widetilde{l}_{1k}\leq \widetilde{\lambda}_{k,r}&=&\sum_{h=0}^{\widetilde{l}_{1k}-1}\sum_{\sigma\in\Omega_h}(p_\sigma s_\sigma^r)^{\frac{\xi_r}{\xi_r+r}}+1+\sum_{\sigma\in\Lambda_{\Gamma_{k,r}}^*}(p_\sigma s_\sigma^r)^{\frac{\xi_r}{\xi_r+r}}\nonumber\\
&\leq&\sum_{h=0}^{\widetilde{l}_{1k}-1}a(\xi_r)^h+1+\sum_{h=\widetilde{l}_{1k}}^{\widetilde{l}_{2k}-1}a(\xi_r)^h=\widetilde{l}_{2k}+2.
\end{eqnarray}
By the definition of $\Gamma_{k,r}$ and $\Gamma_{k,r}(\sigma)$, we deduce
\begin{eqnarray*}
\widetilde{l}_{1k}\leq\phi_{k,r}\cdot\underline{\pi}_r^{\frac{k\xi_r}{\xi_r+r}}\leq (\widetilde{l}_{2k}+2)\cdot \underline{\pi}_r^{-\frac{\xi_r}{\xi_r+r}}\asymp \widetilde{l}_{2k}.
\end{eqnarray*}
From this, we further deduce
\begin{eqnarray}\label{gg66}
\widetilde{l}_{1k}\underline{\pi}_r^{-\frac{k\xi_r}{\xi_r+r}}\lesssim\phi_{k,r}\lesssim \widetilde{l}_{2k}\underline{\pi}_r^{-\frac{k\xi_r}{\xi_r+r}}.
 \end{eqnarray}
 By the definition of $\widetilde{l}_{1k}$ and $\widetilde{l}_{2k}$, one can see that
\begin{eqnarray}\label{gg77}
\underline{\pi}_r^{\widetilde{l}_{1k}}<\underline{\pi}_r^k;\;\;\overline{\pi}_r^{\widetilde{l}_{2k}-1}\geq\underline{\pi}_r^k;\;\;{\rm implying}\;\;\widetilde{l}_{1k},\widetilde{l}_{2k}\asymp k.
\end{eqnarray}
Combining (\ref{gg66}) and (\ref{gg77}), we obtain that
$\log\phi_{k,r}\asymp k$. As is pointed out in the proof for Theorem \ref{mthm2} Case (ii), we have
\[
\widetilde{Q}_{k,r}=\phi_{k,r}^{\frac{r}{\xi_r}}e^r_{\phi_{k,r},r}(\mu)\asymp \widetilde{\lambda}_{k,r}^{\frac{\xi_r+r}{\xi_r}}.
\]
 Using this, and (\ref{ss2}), we get the following estimate:
\[
e^r_{\phi_{k,r},r}(\mu)\asymp\phi_{k,r}^{-\frac{r}{\xi_r}}\widetilde{\lambda}_{k,r}^{\frac{\xi_r+r}{\xi_r}}\asymp\phi_{k,r}^{-\frac{r}{\xi_r}}k^{\frac{\xi_r+r}{\xi_r}}\asymp\phi_{k,r}^{-\frac{r}{\xi_r}}(\log\phi_{k,r})^{\frac{\xi_r+r}{\xi_r}}.
\]
Since $\phi_{k,r}\asymp\phi_{k+1,r}$, as we did in the proof for ISMs in Case (i), by reducing the sequence $(n)_{n=1}^\infty$ to $(\phi_{k,r})_{k=1}^\infty$, one can  obtain (\ref{sanguo}).

\section{Examples and remarks}
\begin{example}\label{eg1}{\rm
We consider the following similitudes on $\mathbb{R}^1$:
\[
f_1(x)=\frac{1}{8}x,\;\;f_2(x)=\frac{1}{8}x+\frac{7}{8}.
\]
Let $(p_i)_{i=0}^2$ be a probability vector satisfying
\[
0<p_0<\delta_0:=1-\exp\big(-2^{-1}(\log 9-\log 8)\big),\;p_1=p_2=2^{-1}(1-p_0).
\]
Let $(t_1,t_2)=(1/3,2/3)$ and let $\mu$ be the ISM in Case (i). Then for large $r$, we have
\begin{eqnarray}\label{gg5}
D_r(\mu)=\xi_r=\xi_{2,r}>\xi_{1,r};\;0<\underline{Q}_r^{\xi_r}(\mu)\leq\overline{Q}_r^{\xi_r}(\mu)<\infty.
\end{eqnarray}
In fact, as pointed out in (14.17) of \cite{GL:00}, by implicit differentiation, we have
\[
\xi'_{1,r}(r)=\frac{\xi_{1,r}}{r}\frac{\sum_{i=1}^2(t_is_i^r)^{\frac{\xi_{1,r}}{\xi_{1,r}+r}}(\log t_i-\xi_{1,r}\log s_i)}{\sum_{i=1}^2(t_is_i^r)^{\frac{\xi_{1,r}}{\xi_{1,r}+r}}(\log t_i+r\log s_i)}.
\]
By \cite{GL:00}, $\xi_{1,r}\to \log 2/\log 8$ as $r\to\infty$. Hence, by L'Hopital's rule, we deduce
\begin{eqnarray*}
&&\lim_{r\to\infty}\bigg(\frac{r}{\xi_{1,r}}\log 2-r\log 8\bigg)
=\lim_{r\to\infty}\frac{\xi'_{1,r}(r)\log 2}{\xi_{1,r}^2r^{-2}}\\&&=\lim_{r\to\infty}\frac{\log 2\sum_{i=1}^2(t_is_i^r)^{\frac{\xi_{1,r}}{\xi_{1,r}+r}}(\log t_i-\xi_{1,r}\log s_i)}{\xi_{1,r}\sum_{i=1}^2(t_is_i^r)^{\frac{\xi_{1,r}}{\xi_{1,r}+r}}(r^{-1}\log t_i+\log s_i)}
\\&&=\frac{1}{2}(\log 9-\log 8).
\end{eqnarray*}
It follows that $2^{-\frac{r}{\xi_{1,r}}}8^{r}\to 1-\delta_0$. Thus, for sufficiently large $r$, we have
\[
p_0<1-2^{-\frac{r}{\xi_{1,r}}}8^{r}\; \;{\rm and}\;\;1-p_0>2^{-\frac{r}{\xi_{1,r}}}8^{r}.
\]
 As in  \cite[Remark 1.4]{zhu:13}, for large $r$, we have
\[
\sum_{i=1}^2(p_is_i^r)^{\frac{\xi_{1,r}}{\xi_{1,r}+r}}=2(2^{-1}(1-p_0)s_i^r)^{\frac{\xi_{1,r}}{\xi_{1,r}+r}}>1.
\]
This implies that $\xi_{1,r}<\xi_{2,r}$ for all large $r$. By Theorem \ref{mthm2}, (\ref{gg5}) follows.

}\end{example}
\begin{example}\label{eg2}{\rm
Let $(f_i)_{i=1}^2,(t_i)_{i=1}^2$ and $(p_i)_{i=0}^2$be as defined in Example \ref{eg1}. We consider the following similitudes on $\mathbb{R}^1$:
\[
g_1(x)=\frac{1}{8}x+\frac{1}{3},\;\;g_2(x)=\frac{1}{8}x+\frac{13}{24}.
\]
We denote by $\mu$ an ISM in Case (ii). One can easily see that the IOSC is satisfied for $\mu$. As is calculated in Example \ref{eg1}, we have $\xi_{1,r}<\xi_{2,r}$ for all large $r$. Hence, by Theorem \ref{mthm2}, we know that $\overline{Q}_r^{\xi_r}(\mu)$ is finite.
}\end{example}

\begin{remark}\label{r1}
By \cite[Corollary 1.2]{zhu:14}, we have $\xi_{1,r}>\xi_{2,r}$ for all sufficiently small $r>0$. In fact, for any $t>0$, we have
\[
\sum_{i=1}^N(p_is_i^r)^{\frac{t}{t+r}}\leq\bigg(\sum_{i=1}^Np_i\bigg)^{\frac{t}{t+r}}\bigg(\sum_{i=1}^Ns_i^t\bigg)^{\frac{r}{t+r}}\to(1-p_0)<1\;(r\to 0).
\]
This implies that $\xi_{2,r}\to 0$ as $r\to 0$, while by Corollary 12.16 of \cite{GL:00}, $\xi_{1,r}$ is bounded from below by the Hausdorff dimension of $\nu$, which is positive.
\end{remark}
\begin{remark}
In Examples \ref{eg1},\ref{eg2}, for all large $r$, we have $\xi_{1,r}<\xi_{2,r}$. By this, Remark \ref{r1} and the intermediate-value theorem, we deduce that there exists an $r$ such that $\xi_{1,r}=\xi_{2,r}$. So, for this $r$, we have, $Q_r^{\xi_r}(\mu)=\infty$. This is in sharp contrast to self-similar measures. Assume that $(f_i)_{i=1}^N$ satisfies the OSC. For a self-similar measure $P$ as defined in (\ref{ssm}), by \cite[Theorem 3.1]{GL:01}, the upper and lower quantization coefficient for $P$ of order $r$ are both positive and finite for all $r>0$.
\end{remark}


\begin{thebibliography}{10}

\bibitem{Bar:88} M.F. Barnsley,
 Fractals everywhere. Academic Press, New York, London, 1988

\bibitem{GL:00} S. Graf and H. Luschgy,
Foundations of quantization for probability distributons, in: Lecture Notes in Math., vol. 1730, Springer, Berlin, 2000.

\bibitem{GL:01}S. Graf and H. Luschgy, Asymptotics of the quantization error for
self-similar probabilities, Real. Anal. Exchange \textbf{26} (2001) 795-810.

\bibitem{GL:04}S. Graf and H. Luschgy, Quantization for probabilitiy
measures with respect to the geometric mean error. \emph{Math. Proc.
Camb. Phil. Soc.} 2004 \textbf{136}, 687--717

\bibitem{GL:05}S. Graf and H. Luschgy,
The point density measure in the quantization of self-similar
probabilities  2005 \emph{Math. Proc. Camb. Phil. Soc.} \textbf{138} 513-531

\bibitem{GN:98}R. Gray and D. Neuhoff,
Quantization, IEEE Trans. Inform. Theory \textbf{44 }(1998) 2325-2383.

\bibitem{Hut:81} J. E. Hutchinson, Fractals and self-similarity.\emph{
Indiana Univ. Math. J.} \textbf{30} 1981, 713-47

\bibitem{KZ:14}M. Kesseb\"{o}hmer and S. Zhu, The upper and lower quantization coefficient for Markov-type measures, Math. Nachr., in press.

\bibitem{Kr:08} W. Kreitmeier, Optimal quantization for dyadic
homogeneous Cantor distributions. Math. Nachr. \textbf{281} (2008), 1307-1327

\bibitem{Olsen:08} L. Olsen and N. Snigireva, Multifractal spectra of in-homogenous self-similar measures.
Indiana Univ. Math. J. \textbf{57} (2008) 1787-1841

\bibitem{PK:01}K. P\"{o}tzelberger, The quantization dimension of distributions,
Math. Proc. Camb. Phil. Soc. \textbf{131} (2001) 507-519.

\bibitem{zhu:14} S. Zhu, The quantization for in-homogeneous self-similar measures with in-homogeneous open set condition, Int. J. Math. \textbf{26} (2015), 1-23.

\bibitem{zhu:13}S. Zhu, Asymptotics of the quantization errors for in-homogeneous self-similar measures supported on self-similar sets, Sci. China Math. \textbf{59} (2016), 337-350.




\end{thebibliography}
\end{document}